\documentclass[11pt,a4paper]{article}
\usepackage[T1]{fontenc}
\usepackage[utf8]{inputenc}
\usepackage{lmodern,microtype}
\usepackage[margin=27mm,headheight=14pt]{geometry}
\usepackage{amsmath,amssymb,amsthm,mathtools}
\usepackage{aliascnt}
\usepackage{enumitem,booktabs,tabularx}
\usepackage[numbers,sort&compress]{natbib}
\usepackage{xcolor}
\definecolor{linkblue}{RGB}{27,70,105}
\usepackage[colorlinks=true,allcolors=linkblue,pdfusetitle]{hyperref}
\usepackage[capitalize,nameinlink,noabbrev]{cleveref}
\usepackage{fancyhdr}
\numberwithin{equation}{section}
\newtheorem{theorem}{Theorem}[section]
\newaliascnt{proposition}{theorem}
\newtheorem{proposition}[proposition]{Proposition}
\aliascntresetthe{proposition}
\newaliascnt{lemma}{theorem}
\newtheorem{lemma}[lemma]{Lemma}
\aliascntresetthe{lemma}
\newaliascnt{corollary}{theorem}
\newtheorem{corollary}[corollary]{Corollary}
\aliascntresetthe{corollary}
\newaliascnt{classical}{theorem}
\newtheorem{classical}[classical]{Classical input}
\aliascntresetthe{classical}
\theoremstyle{definition}
\newaliascnt{definition}{theorem}

\aliascntresetthe{definition}
\theoremstyle{remark}
\newaliascnt{remark}{theorem}
\newtheorem{remark}[remark]{Remark}
\aliascntresetthe{remark}
\crefname{theorem}{theorem}{theorems}
\Crefname{theorem}{Theorem}{Theorems}
\crefname{proposition}{proposition}{propositions}
\Crefname{proposition}{Proposition}{Propositions}
\crefname{lemma}{lemma}{lemmas}
\Crefname{lemma}{Lemma}{Lemmas}
\crefname{corollary}{corollary}{corollaries}
\Crefname{corollary}{Corollary}{Corollaries}
\crefname{classical}{classical input}{classical inputs}
\Crefname{classical}{Classical input}{Classical inputs}
\crefname{definition}{definition}{definitions}
\Crefname{definition}{Definition}{Definitions}
\crefname{remark}{remark}{remarks}
\Crefname{remark}{Remark}{Remarks}
\newcommand{\R}{\mathbb R}

\newcommand{\E}{\mathbb E}
\newcommand{\Prob}{\mathbb P}
\newcommand{\one}{\mathbf 1}
\newcommand{\supp}{\operatorname{supp}}

\newcommand{\W}{\mathcal W}
\newcommand{\src}{\mathsf P}
\newcommand{\norm}[1]{\left\lVert#1\right\rVert}
\newcommand{\ip}[2]{\langle#1,#2\rangle}
\newcommand{\KL}[2]{D\!\left(#1\,\middle\|\,#2\right)}

\title{\textbf{Sudakov Minoration for\\Unconditional Log-Concave Vectors}\\[5pt]
\large Common witnesses, entropy softening, and clipped transport}
\author{Witold Bednorz, Rafa\l{} Martynek and Rafa\l{} Meller%
\thanks{Subject classification: 60G15, 60G17}%
\thanks{Keywords and phrases: Canonical Processes, Invariant Method}%
\thanks{Research partially supported by Grant UMO-2022/47/B/ST1/02114}%
\thanks{Institute of Mathematics, University of Warsaw, Banacha 2,
02-097 Warszawa, Poland}}
\date{}
\hypersetup{
 pdftitle={Sudakov Minoration for Unconditional Log-Concave Vectors},
 pdfauthor={Witold Bednorz, Rafal Martynek, Rafal Meller},
 pdfsubject={Expository revision of Log-concave-2, 8 October 2026},
 pdfkeywords={Sudakov minoration, log-concave measure, unconditionality, entropy, transport}
}
\begin{document}
\maketitle
\begin{center}
\small\textbf{AI was used in this research.}
\end{center}
\begin{abstract}
Sudakov minoration asks whether a large family of separated random
linear forms must have a large expected maximum. We present a proof
candidate for this principle for all unconditional log-concave vectors.
The argument starts with a classical reduction to sparse coordinate
supports. It then selects one feasible threshold witness for each label,
removes a common coordinate core, and controls the remaining overlaps.
The main analytic step is an exponential-moment estimate for a softened
witness payoff. We obtain it by combining a one-dimensional clipping
inequality with triangular transport, using log-concavity of the source
and sign symmetry of the target. A Bernoulli comparison restores the
signs at the end. The exposition includes historical context, a guide to
the proof, all intermediate arguments, and explicit choices of constants.
Two consequences concern covering numbers in the moment metric and
concentration of bounded functions of the magnitudes.
\end{abstract}
\noindent\textbf{Keywords:} Sudakov minoration; unconditional log-concave
measure; joint-tail witness; relative entropy; triangular transport.
\newpage
\tableofcontents
\newpage

\section{Introduction}

\subsection{The question and the main statement}

The principle studied here has a simple interpretation: if a family
contains exponentially many random linear forms, and every two of them
are well separated at the corresponding moment order, then their
expected maximum should be large. The difficulty is that separation is
a statement about each pair, whereas a maximum involves the whole
family at once.

A random vector $X$ in $\R^d$ is \emph{unconditional} if
$(\varepsilon_iX_i)_{i=1}^d$ has the same distribution as $X$ for every
deterministic choice of signs $\varepsilon_i\in\{-1,1\}$. This is also
called $1$-unconditionality. It allows arbitrary dependence between
the magnitudes $|X_i|$. We use log-concavity in the measure sense,
including measures supported on a proper linear subspace. For a finite
set $T\subset\R^d$, write
\[
 \W_X(T)=\E\max_{t\in T}\ip{t}{X},
 \qquad d_P(s,t)=\norm{\ip{t-s}{X}}_P .
\]
Thus $\W_X(T)$ is the expected supremum and $d_P$ is the moment
distance between labels. All constants called universal are
independent of $d,P,X,T$.

\begin{theorem}[Main statement]\label{thm:main}
There is a universal constant $c>0$ with the following property.
If $X$ is unconditional and log-concave, $P\ge1$, $A>0$, and
$T\subset\R^d$ is finite with
\begin{equation}\label{eq:packing}
 |T|\ge e^P,\qquad d_P(s,t)\ge A\quad(s\ne t),
\end{equation}
then
\begin{equation}\label{eq:main-conclusion}
 \W_X(T)\ge cA.
\end{equation}
\end{theorem}

The exponential cardinality in \eqref{eq:packing} matches the moment
order $P$. The proposed conclusion has no further dependence on the
dimension or on the pattern of dependence between the magnitudes.

This manuscript is a working proof candidate for verification. The
selection and transport lemmas below are parts of the proposed
argument, with proofs included. The external results used in the
argument are identified separately in \Cref{sec:inputs}.

\subsection{Historical background}\label{subsec:history}

For a standard Gaussian vector $G$, the classical inequality relates
Euclidean separation to metric entropy:
\[
 \E\max_{t\in T}\ip{t}{G}
 \ge c\,\varepsilon\sqrt{\log |T|}
 \quad\text{if }\norm{t-s}_2\ge\varepsilon\quad(s\ne t).
\]
Sudakov's work on Gaussian measures and $\varepsilon$-entropy
\cite{Sudakov} is the starting point for this line of research.
The moment formulation is natural because
$\norm{\ip{t-s}{G}}_P=\norm{G_1}_P\norm{t-s}_2$ and
$\norm{G_1}_P$ is comparable to $\sqrt P$. For $|T|\ge e^P$,
a Gaussian $L_P$ separation of size $A$ therefore gives a supremum
of order $A$.

Beyond the Gaussian setting, the metric must reflect more than a
covariance matrix. Talagrand proved a Bernoulli minoration
\cite{TalagrandBernoulli} and developed corresponding results for
canonical processes, including coordinates with densities proportional
to $\exp(-|x|^r)$, $r\ge1$ \cite{TalagrandCanonical}.
Lata\l a's work extended the minoration principle to independent
symmetric variables with log-concave tails \cite{Latala1997}.
These results establish the moment formulation in important product
models and supply the Bernoulli input used at the end of our argument.

The extension from product measures to general log-concave measures
was formulated and studied by Lata\l a \cite{LatalaSMP}.
His 2014 paper proves the desired principle for several invariant
classes, including rotationally invariant distributions and uniform
measures on $\ell_r$ balls. For arbitrary unconditional log-concave
vectors it also proves a weaker version with the larger cardinality
assumption $|T|\ge e^{P^2}$. The general question was formulated
independently by Mendelson, Milman, and Paouris; their geometric
program studies a dual minoration through centroid bodies and
dimension reduction \cite{MMP}. The theorem proposed here concerns
the unconditional class at the cardinality scale $e^P$.

Two earlier reductions are especially close to the present proof.
Lata\l a's moment description \cite{LatalaMoments} expresses a sparse
linear form through a feasible joint coordinate threshold.
Bednorz's reduction \cite{Bednorz} replaces a general family of labels
by sparse supports with common nonzero coordinate coefficients.
Later work treats products of radial-type log-concave measures
\cite{BednorzRadial}. Here the sparse reduction is used to expose
the remaining geometric problem: turn witnesses that depend on a
pair of labels into witnesses that can be used simultaneously.

Generalized Orlicz balls provide another useful point of reference.
Pilipczuk and Wojtaszczyk proved negative association of their
coordinate magnitudes \cite{PilipczukWojtaszczyk}.
Wojtaszczyk subsequently gave a simpler proof, and an extension to
related measures, using an argument modeled on geometric localization
\cite{Wojtaszczyk}. These results concern dependence of the source
coordinates. In the present construction, the product-type tail
estimate is instead obtained for the distribution of the selected
witnesses over the label set. This is why negative association does
not appear among the assumptions of the proposed theorem.

The analytic part uses the triangular-transport entropy calculation
underlying the weighted Poincar\'e approach of Cordero-Erausquin and
Gozlan \cite{CEG}. The additional step developed here is to control
bounded coordinate transforms directly by that transport cost.
The resulting estimate is then applied to an entropy-penalized
selection of witnesses. No localization theorem is used as an
external input to this final argument.

\subsection{The proof in five steps}\label{subsec:roadmap}

After the classical reduction, a label is a small set $A_t\subset[d]$
and the process has the form
\[
 X_t=\sum_{i\in A_t}r_iX_i .
\]
A \emph{witness} is a nonnegative vector of thresholds. Its feasibility
means that the magnitudes exceed those thresholds together with
probability at least $e^{-q}$. The witness explains how a pair of
supports can be far apart.

The proof organizes this information as follows.
\begin{enumerate}[leftmargin=*,itemsep=5pt]
\item \textbf{Choose witnesses simultaneously.}
In \Cref{sec:common}, prefix encoders describe the supports.
Convex separation chooses one distribution of encoders whose averages
are feasible witnesses. A graph selection then leaves a large family
on which $a_t(A_t\setminus A_s)$ is large for every ordered pair.

\item \textbf{Remove a common core once.}
Some coordinates may carry a large part of many witnesses.
In \Cref{sec:spread}, a single maximization chooses a subfamily and
a common core. Deleting that core preserves directed separation
and produces a quantitative bound on all joint witness tails.

\item \textbf{Prove concentration for bounded transforms.}
In \Cref{sec:transport}, a clipping inequality controls a
one-dimensional displacement by its entropy cost.
Triangular transport applies this estimate coordinate by coordinate
and gives Gaussian exponential moments for Lipschitz functions of
the transformed magnitudes.

\item \textbf{Pass from feasible events to a typical gain.}
In \Cref{sec:softening}, the capped gains are combined through a
Gibbs variational formula. Softening charges the loss of gain to
label entropy and makes the resulting function Lipschitz in the
coordinates from the previous step. Concentration then gives a
lower bound on its mean.

\item \textbf{Restore the signs.}
In \Cref{sec:signed}, many labels with substantial capped gain
contain a large subfamily with small overlaps. Their signed
coefficient vectors are separated in a Bernoulli moment metric.
Bernoulli minoration gives the required signed supremum.
\end{enumerate}

The common witness remains label-dependent: there is one $a_t$ for
each retained $t$, fixed simultaneously against all other retained
labels. The construction does not require a single threshold vector
shared by the whole family. Also, the core is deleted by projection.
The source is never conditioned on a cell in order to preserve
packing.

For a first reading, \Cref{thm:reduced} gives the intermediate statement
that links the classical reduction to the new argument.
\Cref{thm:spread-minoration} explains how spread and feasibility
combine to prove the lower bound. The coding details are in
\Cref{sec:common}, the analytic details in \Cref{sec:transport,sec:softening},
and the explicit constant choices in \Cref{sec:signed}.
The appendix records the dependencies in one place.

\section{Preliminaries and classical inputs}\label{sec:inputs}

\subsection{Notation and the two probability spaces}

There are two distinct kinds of averaging. The source law describes
the random magnitudes $Y$. The label law $\mu$ describes a finite
collection of deterministic witnesses $a_t$; a posterior $\pi$ is
another probability on that collection. Keeping these averages
separate is useful throughout the proof.

\begin{center}
\small
\begin{tabularx}{\linewidth}{@{}lX@{}}
\toprule
Symbol & Meaning\\
\midrule
$P$ & Moment order in the original packing.\\
$q,p_1,p$ & Logarithms of the numbers of labels after successive
selections; the final family has size $e^p$.\\
$A_t,\ m$ & Support of label $t$ and a common upper bound on its size.\\
$a_t,\ u$ & Feasible threshold witness and its mass scale.\\
$Y,\ \src$ & Magnitudes of the source; in the transport section,
$\src$ denotes the law of the full signed vector.\\
$\mu,\ \pi,\ K_\pi$ & Reference label law, posterior label law, and
relative entropy $D(\pi\Vert\mu)$.\\
$\theta,\ B$ & Threshold shrinkage factor and payoff bound $B=\theta u$.\\
$C_t,\ H_\eta,\ \widehat H_\eta$ & Capped gain, Gibbs free energy,
and softened free energy.\\
$\eta,\ \eta'$ & Temperature and the shifted temperature $\eta'=\eta+b$.\\
\bottomrule
\end{tabularx}
\end{center}

For a vector $a\ge0$ and $D\subset[d]=\{1,\ldots,d\}$, write
$a(D)=\sum_{i\in D}a_i$. For a compact set $K$, its support function is
$h_K(w)=\sup_{a\in K}\ip{w}{a}$. A subset of $\R_+^d$ is
\emph{downward closed} if $0\le b\le a\in K$ implies $b\in K$.
Relative entropy is
\[
 \KL{\rho}{\nu}=\int \log\!\left(\frac{d\rho}{d\nu}\right)d\rho
\]
when $\rho\ll\nu$, and is $+\infty$ otherwise. If $\mu$ is a probability
on a finite label set, we write $K_\mu(\pi)=\KL{\pi}{\mu}$.

\subsection{Signs, projections, and joint tails}
\begin{lemma}\label{lem:signs}
For an unconditional vector $X$, one may realize
$X=(\varepsilon_iY_i)_{i=1}^d$, where $Y=|X|$ and the
$\varepsilon_i$ are independent symmetric signs, independent of $Y$.
Coordinate projections and positive diagonal images of an
unconditional log-concave vector remain unconditional and log-concave.
\end{lemma}
\begin{proof}
Averaging any bounded measurable function over the sign group shows
that its expectation under $X$ equals its expectation under
$(\varepsilon_i|X_i|)_i$. The other assertions follow from invariance
under signs and preservation of log-concavity by linear images.
\end{proof}

Let $Y=|X|$ and define
\[
 Q_Y(v)=\Prob(Y_i\ge v_i\text{ for all }i),\qquad v\in\R_+^d.
\]
The law of $Y$ is log-concave on the positive orthant: in the
nondegenerate case its density is $2^d$ times the restriction of the
unconditional density of $X$. Zero coordinates can be removed in the
degenerate case.

\begin{lemma}[Tail bodies]\label{lem:tails}
The function $Q_Y$ is log-concave and upper semicontinuous, with
$Q_Y(0)=1$. Consequently
\begin{equation}\label{eq:tail-shrink}
 Q_Y(\theta v)\ge Q_Y(v)^\theta,\qquad 0<\theta\le1.
\end{equation}
For $q,u>0$ and $A\subset[d]$, the set
\begin{equation}\label{eq:tail-body}
 K_A(q,u)=\{a\ge0:\supp a\subset A,\ Q_Y(a)\ge e^{-q},
                         \ a([d])\le u\}
\end{equation}
is compact, convex, and downward closed.
\end{lemma}
\begin{proof}
For upper orthants $O_v=v+\R_+^d$,
$(1-\theta)O_v+\theta O_w=O_{(1-\theta)v+\theta w}$.
The defining inequality for a log-concave measure gives log-concavity
of $Q_Y$. If $v_n\to v$, then
$\limsup_n\one_{\{Y\ge v_n\}}\le\one_{\{Y\ge v\}}$;
reverse Fatou proves upper semicontinuity. Interpolating between $0$
and $v$ gives \eqref{eq:tail-shrink}. The superlevel set in
\eqref{eq:tail-body} is closed, convex, and downward closed.
The mass cap makes it bounded.
\end{proof}

\subsection{The three external ingredients}
The following statements are used as classical results. We include
the quantitative form of the support reduction to make the final
choice of constants unambiguous.

\begin{classical}[Sparse moment witness]\label{input:witness}
There is a universal $C_w\ge1$ such that, if $X$ is unconditional
and log-concave, $P\ge1$, $|J|\le P$, and $r_i\ge0$, then there exists
$z\in\R_+^J$ satisfying
\begin{equation}\label{eq:sparse-witness}
 \Prob(|X_i|\ge z_i,\ i\in J)\ge e^{-P},
 \qquad
 \sum_{i\in J}r_i z_i
 \ge C_w^{-1}\norm{\sum_{i\in J}r_iX_i}_P.
\end{equation}
This is the sparse-support case of \cite{LatalaMoments}; see also
\cite[Theorem 1]{Bednorz}.
\end{classical}

\begin{classical}[Bernoulli Sudakov minoration]\label{input:bernoulli}
There is a universal $c_B>0$ such that, whenever $v_t\in\R^d$,
$|T|\ge e^q$, $q\ge1$, and
\[
 \norm{\sum_i(v_{t,i}-v_{s,i})\varepsilon_i}_q\ge b
 \quad(s\ne t),
\]
one has
\[
 \E_\varepsilon\max_{t\in T}|\ip{v_t}{\varepsilon}|\ge c_Bb.
\]
See \cite{TalagrandBernoulli} and
\cite[Example 1.4]{LatalaSMP}. We decrease $c_B$ when harmless.
\end{classical}

\begin{classical}[Sparse binary reduction]\label{input:reduction}
Fix $h,e>0$. There are $c_{h,e}>0$ and
$P_0(h,e)<\infty$ such that the following holds for isotropic
unconditional log-concave $X$ satisfying \eqref{eq:packing}, with
$P\ge P_0(h,e)$. Either $\W_X(T)\ge c_{h,e}A$, or there are
$n=\lceil e^{P/4}\rceil$ distinct supports $A_t$ and weights $r_i\ge0$
such that
\begin{align}
 |A_t|&\le hP,\label{eq:reduction-size}\\
 \norm{\sum_i r_i(\one_{A_t}(i)-\one_{A_s}(i))X_i}_P
      &\ge A/2\quad(s\ne t),\label{eq:reduction-sep}\\
 M_{\rm bin}:=\E\max_t\left|\sum_{i\in A_t}r_iX_i\right|
      &\le 2\W_X(T)+eA.\label{eq:reduction-transfer}
\end{align}
This is the quantitative form of the construction in
\cite[Proposition 2 and Corollary 2]{Bednorz}.
\end{classical}

Here is how its two adjustable parameters arise. Normalize the
original separation to $P$. If the Bernoulli $L_P$ metric has a
large enough packing at radius proportional to $\delta P$, Bernoulli
minoration and the contraction argument in the cited proof give the
first alternative. Otherwise translate a large subset into a
Bernoulli $L_P$ ball of radius $\delta P$. Exponential rounding then
produces a further subset of cardinality at least $e^{P/4}$ and
approximants $\varphi(t)$ with common nonzero coordinate values and
\[
 |\supp\varphi(t)|\le\frac{P}{4\log(1/\rho)},\qquad
 \norm{\ip{t-\varphi(t)}{X}}_P\le C(\rho+\delta)P,
 \qquad \frac{\rho}{\log(1/\rho)}=C'\delta.
\]
Here the common rounded value in coordinate $i$ may initially be a
real number $b_i$. Choose one deterministic sign vector
$\xi_i=\operatorname{sgn}(b_i)$, with $\xi_i=1$ when $b_i=0$.
Replacing $X_i$ by $\xi_iX_i$ and every label coordinate by
$\xi_i\varphi(t)_i$ replaces $b_i$ by $|b_i|$. Unconditionality
preserves the entire joint law of the process. This is a single
coordinatewise sign change shared by all labels, not a different
sign choice for each label.

For completeness, the two approximation requirements can be imposed
simultaneously. Fix a universal $C_{\rm card}$ such that
$n^{1/P}\le C_{\rm card}$ for $n=\lceil e^{P/4}\rceil$ and $P\ge1$;
for example $C_{\rm card}=2\exp(1/4)$ suffices. Given $h,e>0$, choose
$0<\rho<1$ sufficiently small, put
$\delta=\rho/(C'\log(1/\rho))$, and require
\[
 \frac{1}{4\log(1/\rho)}\le h,
 \qquad C(\rho+\delta)
 \le \min\left\{\frac14,\frac{e}{C_{\rm card}}\right\}.
\]
Both requirements hold for all sufficiently small $\rho$, chosen
before $P$. The first gives $|\supp\varphi(t)|\le hP$.
The moment error of each label is at most $P/4$, so the two errors
for a pair leave separation at least $P/2$.
Retain exactly $n$ labels if more have been produced. The bound
\[
 \E\max_{1\le j\le n}|V_j|
 \le n^{1/P}\max_j\norm{V_j}_P
\]
then makes the expected maximum rounding error at most $eP$.
If $t_0$ is the translation anchor, symmetry gives
\[
 \E\max_t|\ip{t-t_0}{X}|
 \le \E\bigl(\max_{t\in T}\ip{t}{X}
                   -\min_{t\in T}\ip{t}{X}\bigr)
 =2\W_X(T).
\]
This proves the transfer estimate in the stated form. In particular,
$h$ and $e$ can be chosen independently before applying the reduction.

We also use the scalar log-concave moment comparison
\begin{equation}\label{eq:scalar}
 \norm{V}_P\le C_{\rm reg}P\,\E|V|,\qquad P\ge1,
\end{equation}
for a centered log-concave real variable; see
\cite[Section 2]{LatalaSMP}. It handles the bounded range of $P$.

\section{Common witnesses from prefix encoders}\label{sec:common}

Pairwise witnesses may point in different coordinate directions.
The aim of this section is to make these choices compatible:
each retained label should have one witness that separates it from
every other retained label. The only ingredients are convexity of
the witness bodies and a counting inequality for support encoders.

This section is entirely deterministic. Let $\mathcal T$ be a finite
label set of size $N\ge2$, put $q=\log N>0$, and let
$A_t\subset[d]$ with $|A_t|\le m$. Suppose that $K_t$ is a compact,
convex, downward closed subset of $\R_+^{A_t}$, containing $0$.
Assume that for some $u>0$, each pair $s\ne t$ satisfies
\begin{equation}\label{eq:H}
 \max\{h_{K_t}(\one_{A_t\setminus A_s}),
        h_{K_s}(\one_{A_s\setminus A_t})\}\ge u.
\end{equation}

The directed alternative \eqref{eq:H} says that one of the two
supports has at least $u$ units of feasible witness mass outside
the other. It does not specify in advance which direction works,
or which point in the corresponding body should be used.

\subsection{An encoder inequality}
An encoder $E$ chooses distinct subsets $C_t\subset A_t$ and uses the
natural coordinate order to write each $C_t$ as an increasing word.
Group labels according to $k=|C_t|$, and let $N_k$ be the size of that
group. Along the word $(i_1,\ldots,i_k)$ for $t$, let $n_j$ be the
number of words in the group having that prefix of length $j$.
Thus $n_0=N_k$ and $n_k=1$. Set
\begin{equation}\label{eq:encoder}
 \tau^E_{t,i_j}=\log(n_{j-1}/n_j),\qquad
 \tau^E_{t,i}=0\quad(i\notin C_t).
\end{equation}
All entries are nonnegative, and
$\tau^E_t([d])=\log N_k\le q$. The empty-word group has at most one
label and contributes the zero vector.

The charge $\tau^E_{t,i_j}$ measures the logarithmic reduction in
the number of possible labels when the next coordinate is revealed.
For example, the four words
\[
 (1,3),\ (1,4),\ (2,3),\ (2,4)
\]
have two charges $\log2$ each. Each full word has weight $1/4$.
If coordinates in $S=\{1,2\}$ are omitted from the charge, each word
has weight $1/2$, and their total weight becomes $2$.
The next lemma bounds this increase when at most $m$ coordinates
are omitted. This is the counting estimate needed for comparison
with another support.

\begin{lemma}[Prefix counting]\label{lem:prefix}
For every $S\subset[d]$ with $|S|\le m$ and every encoder,
\begin{equation}\label{eq:prefix}
 \sum_{t\in\mathcal T}
       \exp\{-\tau^E_t(A_t\setminus S)\}
 \le B_m,\qquad B_m=(m+1)4^m.
\end{equation}
\end{lemma}
\begin{proof}
Fix a word length $k$. Classify its words by the set of positions
$J\subset[k]$ occupied by coordinates in $S$, and by the set
$D=C_t\cap S$. There are at most $2^k2^{|S|}$ classes.
Within a class the increasing order of coordinates fixes the entries
at positions $J$.

Consider a random walk down the prefix tree of the entire length-$k$
group. At a position in $J$, force the prescribed entry, stopping if
that extension is absent. At every other position choose the next
entry with its original conditional probability $n_j/n_{j-1}$.
The probability of arriving at a given word in the class is
\[
 \prod_{j\notin J}\frac{n_j}{n_{j-1}}
 =\exp\{-\tau^E_t(C_t\setminus S)\}
 =\exp\{-\tau^E_t(A_t\setminus S)\}.
\]
Different complete words are disjoint outcomes, so their
probabilities sum to at most one. Sum over classes and then over
$0\le k\le m$ to obtain
$\sum_{k=0}^m2^k2^{|S|}\le(m+1)4^m$.
\end{proof}

\subsection{One distribution of encoders}

The encoder inequality holds for every encoder, but a particular
encoder need not produce feasible thresholds. We average over
encoders and use convexity to choose an average that lies in all
the required witness bodies at once. The separating functional
in the proof can be interpreted as a system of coordinate costs.

\begin{proposition}[Simultaneous witnesses]\label{prop:common}
Under \eqref{eq:H}, there is a probability distribution on encoders
such that
\begin{equation}\label{eq:common-a}
 a_t=\frac{u}{2q}\E_E\tau^E_t\in K_t
 \quad\text{for every }t.
\end{equation}
These witnesses satisfy $a_t([d])\le u/2$ and, for each $s$,
\begin{equation}\label{eq:common-count}
 \sum_t\exp\!\left\{-\frac{2q}{u}a_t(A_t\setminus A_s)\right\}
 \le B_m.
\end{equation}
\end{proposition}
\begin{proof}
Let $\mathcal C$ be the convex hull of all arrays
$(\tau_t^E)_{t\in\mathcal T}$, and put
$\mathcal K=\prod_tK_t$. There are only finitely many encoders.
We show that $(u/(2q))\mathcal C$ meets $\mathcal K$.

\emph{Step 1: find distinct inexpensive cores.}
Fix nonnegative costs $w_t$ and put $H_t=h_{K_t}(w_t)$. Define
\[
 C_t=\{i\in A_t:w_{t,i}\le2H_t/u\}.
\]
These subsets are distinct. Indeed, if the first alternative in
\eqref{eq:H} holds, take $v\in K_t$, supported on
$A_t\setminus A_s$, with $v([d])\ge u$; downward closure allows the
restriction to that difference. If $H_t>0$, the mass of $v$ outside
$C_t$ is at most $u/2$, since $\ip{w_t}{v}\le H_t$.
If $H_t=0$, all its positive mass lies where $w_{t,i}=0$.
In either case $C_t\setminus A_s$ is nonempty, so
$C_t\ne C_s\subset A_s$. The other alternative is symmetric.

\emph{Step 2: bound the cost of the encoder.}
For the resulting encoder,
\[
 \ip{w_t}{\tau_t^E}
 \le\frac{2H_t}{u}\tau_t^E([d])
 \le\frac{2q}{u}H_t.
\]
Consequently
\begin{equation}\label{eq:sep-dual}
 \min_{b\in (u/(2q))\mathcal C}\sum_t\ip{w_t}{b_t}
 \le\sum_t h_{K_t}(w_t)
 \qquad(w_t\ge0).
\end{equation}
\emph{Step 3: apply convex separation.}
If the two compact convex sets were disjoint, a separating
functional would make the opposite strict inequality hold.
It can be chosen nonnegative: replacing each coefficient $w_{t,i}$
by its positive part leaves the support function of the downward
closed set $\mathcal K$ unchanged and can only increase the
functional on the nonnegative set $\mathcal C$. This contradicts
\eqref{eq:sep-dual}.

An intersection point gives \eqref{eq:common-a}. The mass bound
follows from $\tau_t^E([d])\le q$. Jensen's inequality and
\Cref{lem:prefix}, with $S=A_s$, give
\[
 \sum_t e^{-(2q/u)a_t(A_t\setminus A_s)}
 \le \E_E\sum_t e^{-\tau_t^E(A_t\setminus A_s)}
 \le B_m.
\]
\end{proof}

\begin{corollary}[Directed separation on a large subfamily]
\label{cor:directed}
If $q\ge64$ and $m\le q/64$, there is a subfamily
$\mathcal T_1$ of size at least $e^{q/2}$ such that
\begin{equation}\label{eq:directed}
 a_t(A_t\setminus A_s)\ge u/8
 \quad(s,t\in\mathcal T_1,\ s\ne t).
\end{equation}
\end{corollary}
\begin{proof}
For fixed $s$, \eqref{eq:common-count} bounds the number of $t$
with $a_t(A_t\setminus A_s)<u/8$ by
$D=B_me^{q/4}$. Join two distinct labels by an undirected edge
if either directed inequality fails. This graph has at most $ND$
edges. The elementary independent-set bound
$\alpha(G)\ge N^2/(N+2|E(G)|)$ gives an independent set of size at
least $N/(1+2D)$.

For $m\ge1$, $\log B_m\le m\log8$ because $m+1\le2^m$.
Hence, using $q\ge64$,
\[
 \log(1+2D)\le\log3+\frac q4+\frac{q\log8}{64}
 \le q/2.
\]
The case $m=0$ cannot satisfy \eqref{eq:H} for distinct labels.
\end{proof}

\section{A single core selection and witness spread}\label{sec:spread}

The preceding selection guarantees separation, but the witnesses may
still overlap heavily. A common coordinate core causes no separation
between labels, so it can be removed. For supports of the form
$A_t=I\cup B_t$, the signed sum over $I$ is a common random term;
averaging over its signs and applying Jensen cannot decrease the
expected residual maximum.

The issue is to choose the core without spending too many labels.
The objective below rewards a large retained family, the number of
common coordinates, and the amount of common witness mass.
At a maximizer, any further common requirement must remove a
controlled proportion of the remaining labels. This gives all the
spread inequalities in one step.

Suppose now that a family $\mathcal T_1$ of size $e^{p_1}$ has
supports of size at most $m\ge1$ and witnesses $a_t$ of mass at most
$u$. The witnesses are fixed throughout this selection.

\begin{proposition}[Common-core optimization]\label{prop:spread}
There are a subfamily $\mathcal S\subset\mathcal T_1$,
$|\mathcal S|\ge e^{3p_1/4}$, and a set
$I\subset\bigcap_{t\in\mathcal S}A_t$ such that, after deleting
coordinates in $I$, the uniform label law $\mu$ satisfies
\begin{equation}\label{eq:spread-strong}
 \mu\{J\subset A_t,\ a_{t,j}\ge v_j\ (j\in J)\}
 \le \exp\!\left\{-\frac{p_1}{8m}|J|
                     -\frac{p_1}{8u}\sum_{j\in J}v_j\right\}
\end{equation}
for every $J\subset I^c$ and $v_j\ge0$.
Any directed separation \eqref{eq:directed} is preserved.
\end{proposition}
\begin{proof}
For $I\subset[d]$ and $v\in\R_+^I$, set
\[
 \mathcal S(I,v)=\{t:I\subset A_t,\ a_{t,i}\ge v_i\ (i\in I)\}.
\]
Among nonempty such sets, maximize
\[
 \Phi(I,v)=\log|\mathcal S(I,v)|
                +\tau_1|I|+\lambda_1\sum_{i\in I}v_i,
 \qquad \tau_1=\frac{p_1}{8m},\quad\lambda_1=\frac{p_1}{8u}.
\]
A maximum exists: each threshold can be raised to one of the
finitely many witness values without changing its nonempty selected
set. For nonempty $\mathcal S(I,v)$ one has $|I|\le m$ and
$\sum v_i\le u$. The empty core gives $\Phi=p_1$; thus a maximizer
satisfies $\log|\mathcal S|\ge3p_1/4$.

For an additional constraint $(J,w)$ outside $I$, maximality gives
\[
 \log|\mathcal S(I\cup J,v\cup w)|
       +\tau_1|J|+\lambda_1\sum_{j\in J}w_j
 \le\log|\mathcal S(I,v)|.
\]
If the set on the left is empty the desired estimate is automatic;
otherwise exponentiation proves \eqref{eq:spread-strong}.
Finally, common coordinates belong to both supports, so deleting
them changes none of the quantities $a_t(A_t\setminus A_s)$.
\end{proof}

Let $p=\log|\mathcal S|$. Since $p\le p_1$, the weaker parameters
\begin{equation}\label{eq:spread-parameters}
 \tau=\frac{p}{8m},\qquad \lambda=\frac{p}{8u},
 \qquad q_{\rm sp}=e^{-\tau}
\end{equation}
may replace $\tau_1,\lambda_1$. In particular
\begin{equation}\label{eq:S}
 \mu\{a_{t,i}>v_i\text{ for every }i\in J\}
 \le q_{\rm sp}^{|J|}e^{-\lambda\sum_{i\in J}v_i},
 \qquad v_i\ge0.
\end{equation}
Strict thresholds in \eqref{eq:S} exclude zero witness coordinates.
On a subfamily of relative size at least $\alpha$, the same
inequality holds with an additional factor $\alpha^{-1}$.

\begin{lemma}[Deleting a common core]\label{lem:projection}
For $Y=|X|$ with unconditional $X$, deleting coordinates common to
all supports can only decrease
\[
 \E_Y\E_\varepsilon
       \max_t\left|\sum_{i\in A_t}\varepsilon_iY_i\right|.
\]
Projection also preserves log-concavity, and a joint-tail witness
remains feasible after coordinates are deleted.
\end{lemma}
\begin{proof}
Condition on $Y$ and the signs outside the common core $I$.
The expression inside the maximum is
$U+V_t$, where $U=\sum_{i\in I}\varepsilon_iY_i$ is the same for
all $t$ and has conditional mean zero. Convexity gives
$\E_{\varepsilon_I}\max_t|U+V_t|\ge\max_t|V_t|$.
For joint tails, deleting constraints enlarges the event. The
remaining assertion follows from \Cref{lem:signs}.
\end{proof}

\begin{lemma}[A fixed-support overlap estimate]\label{lem:overlap}
Under \eqref{eq:S}, for any $D\subset[d]$ with $|D|\le m$,
\begin{equation}\label{eq:overlap-mgf}
 \E_\mu e^{(\lambda/2)a_t(D)}
 \le(1+q_{\rm sp})^{|D|}\le e^{m q_{\rm sp}}.
\end{equation}
Consequently, for $\gamma>0$,
\begin{equation}\label{eq:overlap-tail}
 \mu\{a_t(D)\ge\gamma u\}
 \le\delta,\qquad
 \delta=\exp\{m q_{\rm sp}-\gamma p/16\}.
\end{equation}
\end{lemma}
\begin{proof}
Expand the product
$\prod_{i\in D}(1+(e^{\lambda a_{t,i}/2}-1))$.
For each $J\subset D$, Tonelli and \eqref{eq:S} bound the
expectation of the corresponding term by
\[
 \int_{\R_+^J}\prod_{i\in J}
       \left(\frac{\lambda}{2}e^{\lambda v_i/2}\right)
       q_{\rm sp}^{|J|}e^{-\lambda\sum v_i}\,dv
 =q_{\rm sp}^{|J|}.
\]
Summing proves \eqref{eq:overlap-mgf}; exponential Markov and
$\lambda u/2=p/16$ give \eqref{eq:overlap-tail}.
\end{proof}

\section{Clipped transport and bounded-coordinate concentration}
\label{sec:transport}

This section supplies the concentration estimate used to turn
rare feasible witness events into a positive mean gain.
The functions to be concentrated become Lipschitz after replacing
each magnitude $y_i$ by $1-e^{-a_i y_i}$, which takes values in
$[0,1]$. We prove a Gaussian exponential-moment bound for precisely
these functions.

The route has three parts. First, clipping a one-dimensional
transport displacement makes its squared size controllable by the
convex cost $\Delta(v)=v-1-\log v$. Second, conditional sign
symmetry lets us sum that inequality along a triangular transport.
Finally, applying the resulting entropy estimate to an exponential
tilt gives concentration at every value of the exponential parameter.
Only the source is required to be log-concave.

\subsection{A one-dimensional clipping inequality}

Put
\[
 \Delta(v)=v-1-\log v\quad(v>0),\qquad
 c_R(x)=\min(x/R,1)\quad(x\ge0,\ R>0).
\]
Set $\Delta(0)=+\infty$.

\begin{lemma}[Clipping]\label{lem:clipping}
There is a universal $C_{\rm clip}<\infty$ such that, for every
nonincreasing probability density $h$ on $\R_+$ and every
nondecreasing locally absolutely continuous map
$T:\R_+\to\R_+$ with $T(0)=0$,
\begin{equation}\label{eq:clipping}
 \int_0^\infty |c_R(T(x))-c_R(x)|^2h(x)\,dx
 \le C_{\rm clip}\int_0^\infty\Delta(T'(x))h(x)\,dx.
\end{equation}
\end{lemma}

\begin{proof}
\emph{Step 1: an initial interval.}
We first prove the unweighted estimate on every interval $[0,s]$.
Scaling $x$ and $T$ by $R$ reduces to $R=1$. Write
$E_s=\int_0^s\Delta(T'(x))\,dx$ and assume $E_s<\infty$.
The elementary inequality
\begin{equation}\label{eq:delta-elementary}
 |v-1|\le C\{\sqrt{\Delta(v)}+\Delta(v)\},\qquad v>0,
\end{equation}
follows by considering $v\in[1/2,2]$ and its complement.
Absolute continuity, $T(0)=0$, and Cauchy--Schwarz imply
\[
 |T(x)-x|\le C\{\sqrt{xE_s}+E_s\},\qquad 0\le x\le s.
\]
For $x\le\min(s,2)$, the clipping difference is bounded by both
$1$ and $|T(x)-x|$. If $E_s\le1$, its square is at most $CE_s$;
if $E_s>1$, the same conclusion follows from the bound by $1$.
Integration over this part of the interval contributes at most
$CE_s$.

\emph{Step 2: the remaining part of the interval.}
For $x>2$, a nonzero clipping difference requires $T(x)<1$.
If this set is nonempty, let
$b=\sup\{x\in[2,s]:T(x)<1\}$. When $b>2$, continuity gives
$T(b)\le1$. Jensen's inequality gives
\[
 E_s\ge\int_0^b\Delta(T'(x))\,dx
 \ge b\,\Delta\!\left(\frac{T(b)}b\right)
 \ge b\,\Delta(1/2).
\]
Here $\Delta$ is decreasing on $(0,1]$, and zero average derivative
would give infinite cost. The remaining clipping integral is at
most the length of the exceptional interval, hence at most $b$.
This proves
\begin{equation}\label{eq:interval-clipping}
 \int_0^s|c_R(T(x))-c_R(x)|^2\,dx
 \le C_{\rm clip}\int_0^s\Delta(T'(x))\,dx
\end{equation}
for every $s,R>0$.

\emph{Step 3: pass to a decreasing density.}
Every nonincreasing integrable density is a mixture of
uniform densities on initial intervals. More precisely, using a
right-continuous version and Stieltjes differentiation, the measure
$\omega(ds)=-s\,dh(s)$ has total mass one and
\[
 h(x)=\int_{[x,\infty)}\frac1s\,\omega(ds)
 \quad\text{for almost every }x>0.
\]
The boundary terms vanish because $s h(s)\to0$ at both endpoints;
alternatively the identity and normalization follow directly by
Tonelli. Integrate \eqref{eq:interval-clipping}, divided by $s$,
against $\omega$. Nonnegativity justifies all interchanges, including
when the right-hand side is infinite.
\end{proof}

\begin{lemma}[Exponential saturation]\label{lem:saturation}
For $a>0$, let $\psi_a(x)=1-e^{-ax}$ on $\R_+$. Under the hypotheses
of \Cref{lem:clipping},
\begin{equation}\label{eq:saturation}
 \int|\psi_a(T(x))-\psi_a(x)|^2h(x)\,dx
 \le C_{\rm clip}\int\Delta(T'(x))h(x)\,dx .
\end{equation}
\end{lemma}
\begin{proof}
Direct integration gives the probability-mixture representation
\[
 \psi_a(x)=\int_0^\infty c_R(x)\,a^2R e^{-aR}\,dR,
 \qquad \int_0^\infty a^2R e^{-aR}\,dR=1.
\]
Apply Jensen to the square of the difference, then use
\Cref{lem:clipping} and Tonelli.
\end{proof}

Absolute continuity in these lemmas is substantive. A monotone map
with a jump at the origin is not covered by the derivative cost.
The approximation in the next subsection ensures that the
conditional maps used in the proof have no such jump.

\subsection{Triangular transport with a symmetric target}

For $a_i>0$, define
\[
 \Psi(y)=(1-e^{-a_i y_i})_{i=1}^d,\qquad y\in\R_+^d.
\]
We first work with a smooth, strictly positive unconditional
log-concave source density $f=e^{-V}$ on $\R^d$. It is enough to
consider sources obtained by convolving a log-concave law with a
nondegenerate Gaussian; the final approximation will use exactly
these sources.

\begin{lemma}[Entropy cost for transformed magnitudes]
\label{lem:transport-cost}
Let $\src$ be such a source law and let $\bar\rho$ be unconditional.
If $\bar\rho$ has a smooth density, is supported on a box centered
at the origin, and has a positive density in the interior of that
box, there is a coupling $(X,\widetilde X)$ of $\src,\bar\rho$ such that
\begin{equation}\label{eq:transport-cost}
 \E\norm{\Psi(|\widetilde X|)-\Psi(|X|)}_2^2
 \le C_{\rm clip}\KL{\bar\rho}{\src}.
\end{equation}
The same conclusion holds when
$d\bar\rho/d\src$ is bounded, continuous, strictly positive, and
unconditional.
\end{lemma}

\begin{proof}
Use the increasing triangular (Knothe) transport $T$ from $\src$ to
$\bar\rho$: $T_i$ depends on $(x_1,\ldots,x_i)$ and, for each
preceding prefix, is the increasing rearrangement of the conditional
law of $X_i$ onto the corresponding target conditional law.
For smooth positive densities on the indicated interiors these
one-dimensional maps are locally absolutely continuous and strictly
increasing.

\emph{Step 1: the entropy controls diagonal derivatives.}
The above-tangent entropy calculation gives
\begin{equation}\label{eq:above-tangent}
 \KL{\bar\rho}{\src}\ge
 \sum_{i=1}^d\E_{\src}\Delta(\partial_iT_i(X)).
\end{equation}
We recall the calculation; its standard transport form also appears
in \cite[Lemma 3.5]{CEG}. Change of variables and triangularity give
\[
 \KL{\bar\rho}{\src}
 =\E_{\src}\!\left[V(TX)-V(X)-\sum_i\log\partial_iT_i(X)\right].
\]
Convexity of $V$ bounds its increment below by
$\sum_i\partial_iV(X)(T_i(X)-X_i)$. Integration by parts gives
\[
 \E_{\src}[\partial_iV(X)(T_i(X)-X_i)]
 =\E_{\src}[\partial_iT_i(X)-1],
\]
and substitution proves \eqref{eq:above-tangent}.
These integrations may first be made with compact cutoffs.
The target coordinates are bounded. For a Gaussian-convolved
source, the score is a conditional expectation of a Gaussian
increment divided by its variance; its first moments, and the
moments needed after multiplication by $|X_i|$, are finite.
Thus the cutoff terms vanish. Nonnegativity of $\partial_iT_i$
also allows truncation in the derivative integrals.

\emph{Step 2: apply the one-dimensional estimate conditionally.}
Fix a preceding source prefix $x_{<i}$. The source conditional
density of $X_i$ is even and log-concave: integrate out later
coordinates, use preservation of log-concavity by marginals, and
then restrict to the fixed prefix. Its magnitude density $h_i$
on $\R_+$ is therefore nonincreasing. The corresponding target
conditional density is even, by unconditionality of $\bar\rho$.
Consequently $T_i(x_{<i},\cdot)$ is odd, fixes zero, and its
restriction to $\R_+$ satisfies the hypotheses of
\Cref{lem:saturation}. Conditionally,
\[
 \E\!\left[
  |\psi_{a_i}(|T_i(X)|)-\psi_{a_i}(|X_i|)|^2\mid X_{<i}
 \right]
 \le C_{\rm clip}
 \E[\Delta(\partial_iT_i(X))\mid X_{<i}].
\]
Sum and integrate, using \eqref{eq:above-tangent}, to obtain
\eqref{eq:transport-cost}.

\emph{Step 3: approximate a bounded density change.}
For the final assertion write $r=d\bar\rho/d\src$.
Let $\chi_R$ be a smooth unconditional cutoff, with
$0\le\chi_R\le1$, equal to one on $[-R,R]^d$, supported on
$[-2R,2R]^d$, and strictly positive in the interior of the latter
box. Put
\[
 Z_R=\int r\chi_R\,d\src,\qquad
 r_R=\frac{r\chi_R}{Z_R},\qquad
 d\bar\rho_R=r_R\,d\src.
\]
Since $r$ is bounded and $\int r\,d\src=1$, dominated convergence
gives $Z_R\to1$ and $r_R\to r$ in $L^1(\src)$.
For all large $R$, $0\le r_R\le2\norm{r}_\infty$.
The function $z\mapsto z\log z$, extended by zero at $z=0$,
is continuous and bounded on this fixed compact interval.
Consequently
\[
 \KL{\bar\rho_R}{\src}
 =\int r_R\log r_R\,d\src
 \longrightarrow \int r\log r\,d\src
 =\KL{\bar\rho}{\src}.
\]

To obtain smooth densities, fix $R$ and take a nonnegative smooth
approximate identity $\kappa_\epsilon$ invariant under all
coordinate sign changes. Replace $r$ by
$r_\epsilon=r*\kappa_\epsilon$ and set
\[
 Z_{R,\epsilon}=\int r_\epsilon\chi_R\,d\src,
 \qquad r_{R,\epsilon}
 =\frac{r_\epsilon\chi_R}{Z_{R,\epsilon}}.
\]
The functions $r_\epsilon$ are smooth, unconditional, strictly
positive, and bounded above by $\norm{r}_\infty$. They converge
uniformly to $r$ on the fixed box. In particular,
$Z_{R,\epsilon}\to Z_R>0$, and
$r_{R,\epsilon}\to r_R$ uniformly on that box.
The target density $r_{R,\epsilon}f$ is smooth, supported on the
centered box, and positive in its interior. It therefore satisfies
the first part of the lemma. Uniform convergence and continuity
of $z\log z$ on a bounded interval show convergence of its entropy
to that of $\bar\rho_R$. Choose a diagonal sequence
$R\to\infty$, $\epsilon\downarrow0$ for which both the laws and
the entropies converge to those of $\bar\rho$.

Each approximating pair admits a coupling satisfying
\eqref{eq:transport-cost}. The fixed first marginal and the
convergent second marginals make these couplings tight. Along a
weakly convergent subsequence the limiting marginals are
$\src,\bar\rho$, and the transport cost converges because
$\norm{\Psi(|x'|)-\Psi(|x|)}_2^2$ is bounded and continuous.
Passing to the limit proves the assertion.
\end{proof}

\begin{theorem}[Concentration after saturation]\label{thm:saturated-mgf}
Let $X$ be any unconditional log-concave vector, let $a_i>0$, and
let $g:[0,1]^d\to\R$ be $L$-Lipschitz for Euclidean distance.
Then, for every $t\in\R$,
\begin{equation}\label{eq:saturated-mgf}
 \log\E\exp\!\left\{t\bigl(g(\Psi(|X|))
                   -\E g(\Psi(|X|))\bigr)\right\}
 \le C_T L^2t^2,\qquad C_T=C_{\rm clip}/4 .
\end{equation}
The constant is independent of the coordinate parameters $a_i$.
\end{theorem}

\begin{proof}
First take a Gaussian-convolved source as above, and write
$F(x)=g(\Psi(|x|))$. This is bounded, continuous, and unconditional.
Its exponential tilt $\src_t$, defined by
$d\src_t/d\src=e^{tF}/\E_{\src}e^{tF}$, satisfies the last assertion of
\Cref{lem:transport-cost}. Write $D_t=\KL{\src_t}{\src}$. In the coupling
from that lemma, Lipschitz continuity and Cauchy--Schwarz give
\[
 |\E_{\src_t}F-\E_{\src}F|\le L\sqrt{C_{\rm clip}D_t}.
\]
The entropy identity therefore yields
\begin{align*}
 \log\E_{\src}e^{t(F-\E_{\src}F)}
 &=t(\E_{\src_t}F-\E_{\src}F)-D_t\\
 &\le |t|L\sqrt{C_{\rm clip}D_t}-D_t
 \le \frac{C_{\rm clip}}4L^2t^2.
\end{align*}
The last inequality maximizes a quadratic polynomial in $\sqrt{D_t}$.

For a general source take $X_s=X+\sqrt{s}\,G$ with independent
standard Gaussian $G$. The laws remain unconditional and
log-concave and have the required smooth positive densities.
As $s\downarrow0$, bounded continuity of $F$ gives convergence
of its mean and of every bounded exponential $e^{tF}$.
The constant in the inequality is independent of $s$, so passage
to the limit proves \eqref{eq:saturated-mgf}, also for singular
sources.
\end{proof}

\begin{remark}\label{rem:transport-scope}
This proof does not assert that $\src_t$ is log-concave. Its
unconditionality ensures that its conditional transport maps are
odd; convexity is used only for the potential of the source.
The conclusion concerns functions of \emph{magnitudes}. It does
not give the same Euclidean transport assertion for arbitrary
functions of a signed, coordinatewise saturated vector.
\end{remark}

\section{Entropy softening and the free-energy estimate}
\label{sec:softening}

Feasibility gives a lower bound on an exponential average of the
capped gains. The final signed argument, however, needs a gain that
is present on typical source realizations. These are different
requirements: for a positive partition function $Z$, a lower bound
on $\log\E Z$ alone does not give a lower bound on $\E\log Z$.

We address the difference through a Gibbs selector. A label
distribution $\pi$ earns its average capped gain and pays relative
entropy with respect to $\mu$. We soften its coordinate contributions
so that they have controlled derivatives, prove that the lost gain
is charged to the same entropy, and then apply the concentration
theorem from the previous section.

Let $\mu$ be a label law and suppose, more generally, that
\begin{equation}\label{eq:hereditary-spread}
 \mu\{a_{t,i}>v_i,\ i\in J\}
 \le \alpha^{-1}q_{\rm sp}^{|J|}
             e^{-\lambda\sum_{i\in J}v_i},
 \qquad 0<\alpha\le1,
\end{equation}
with $q_{\rm sp}=e^{-\tau}<1$. Each $a_t$ has support of size at
most $m$ and mass at most $u$. For the main proof $\mu$ is uniform,
$\alpha=1$, and the parameters are \eqref{eq:spread-parameters}.

Fix $0<\theta\le1$, put $B=\theta u$, and define the capped gain
\begin{equation}\label{eq:capped}
 C_t(y)=\sum_i\min(\theta a_{t,i},y_i),\qquad y\ge0.
\end{equation}
Thus $0\le C_t\le B$. Coordinate $i$ contributes up to its requested
amount $\theta a_{t,i}$; if the source magnitude is smaller, only
the available amount $y_i$ is counted. For a posterior label law
$\pi\ll\mu$ put
$Q_i^\pi(v)=\pi\{a_{t,i}>v\}$ and $K_\pi=\KL{\pi}{\mu}$.
Define its softened energy by
\begin{equation}\label{eq:soft-energy}
 \mathcal A_y(\pi)
 =\theta\sum_i\int_0^{y_i/\theta}
       \min\{Q_i^\pi(v),\sqrt{q_{\rm sp}}e^{-\lambda v/2}\}\,dv.
\end{equation}

\subsection{Charging the loss to entropy}

The envelope in \eqref{eq:soft-energy} limits the influence of a
coordinate on the selected gain. If a posterior puts more mass than
this envelope on many large witness coordinates, it must differ
substantially from the spread reference law. The next lemma makes
this cost quantitative.

\begin{lemma}[Entropy softening]\label{lem:softening}
Under \eqref{eq:hereditary-spread},
\begin{equation}\label{eq:soft-loss}
 0\le\E_\pi C_t(y)-\mathcal A_y(\pi)
 \le b(K_\pi+\log\alpha^{-1}),\qquad
 b=\frac{2\theta}{\lambda(1-\sqrt{q_{\rm sp}})}.
\end{equation}
\end{lemma}

\begin{proof}
Let $\zeta$ be a nonnegative variable with
$\Prob(\zeta>v)=q_{\rm sp}e^{-\lambda v}$ for $v\ge0$.
For nonnegative nondecreasing absolutely continuous functions
$\phi_i$ with $\phi_i(0)=0$, expansion of the product
$\prod_i(1+(e^{\phi_i(a_{t,i})}-1))$, followed by joint-tail
integration, gives
\begin{equation}\label{eq:label-mgf-domination}
 \E_\mu e^{\sum_i\phi_i(a_{t,i})}
 \le \alpha^{-1}\prod_i\E e^{\phi_i(\zeta)}.
\end{equation}
Indeed, each nonempty product term is bounded using
\eqref{eq:hereditary-spread}; its integral factors. The empty term
is also bounded by $\alpha^{-1}\ge1$. Truncation removes any
initial boundedness assumptions on the functions.

Entropy duality, applied to the sum of these functions, now gives
\begin{equation}\label{eq:entropy-test}
 \sum_i\left(\E_\pi\phi_i(a_{t,i})
                   -\log\E e^{\phi_i(\zeta)}\right)
 \le K_\pi+\log\alpha^{-1}.
\end{equation}
For each $i$, let
\[
 E_i=\{v\ge0:Q_i^\pi(v)>
                    \sqrt{q_{\rm sp}}e^{-\lambda v/2}\},
 \qquad
 \phi_i(x)=\frac{\lambda}{2}\int_0^x\one_{E_i}(v)\,dv.
\]
Since $\phi_i(v)\le\lambda v/2$, tail integration and
$\log(1+x)\le x$ imply
\begin{align*}
 \log\E e^{\phi_i(\zeta)}
 &\le \frac{\lambda}{2}\int_{E_i}
          q_{\rm sp}e^{-\lambda v/2}\,dv\\
 &\le \sqrt{q_{\rm sp}}\,
             \frac{\lambda}{2}\int_{E_i}Q_i^\pi(v)\,dv
 =\sqrt{q_{\rm sp}}\,\E_\pi\phi_i(a_{t,i}).
\end{align*}
Insert this in \eqref{eq:entropy-test}. It follows that
\[
 \sum_i\int_{E_i}Q_i^\pi(v)\,dv
 \le\frac{2(K_\pi+\log\alpha^{-1})}
           {\lambda(1-\sqrt{q_{\rm sp}})}.
\]
Finally
$\E_\pi C_t(y)=\theta\sum_i\int_0^{y_i/\theta}Q_i^\pi(v)\,dv$.
The loss incurred by taking the minimum in
\eqref{eq:soft-energy} is bounded by $\theta$ times the last
display, proving \eqref{eq:soft-loss}.
\end{proof}

For $\eta>0$, let
\begin{align}
 H_\eta(y)&=\eta\log\E_\mu e^{C_t(y)/\eta}
   =\sup_\pi\{\E_\pi C_t(y)-\eta K_\pi\},\label{eq:Heta}\\
 \widehat H_\eta(y)&=\sup_\pi
              \{\mathcal A_y(\pi)-\eta K_\pi\}.\label{eq:Hhat}
\end{align}
For completeness, the optimizing distribution in \eqref{eq:Heta} is
\[
 \pi_{\eta,y}(t)=
 \frac{\mu(t)e^{C_t(y)/\eta}}{\E_\mu e^{C_t(y)/\eta}}.
\]
For every $\pi\ll\mu$,
$\E_\pi C_t(y)-\eta K_\pi
=H_\eta(y)-\eta D(\pi\Vert\pi_{\eta,y})$.
Nonnegativity of relative entropy proves the variational identity.
This also explains the temperature: decreasing $\eta$ makes the
selector favor larger gains more strongly.
The suprema may be restricted to the simplex on labels of positive
$\mu$-mass. Compactness and continuity give measurability; the
Lipschitz argument below also gives continuity.
\Cref{lem:softening} immediately yields
\begin{equation}\label{eq:Hcompare}
 0\le\widehat H_\eta\le H_\eta\le B,\qquad
 H_{\eta+b}-b\log\alpha^{-1}\le\widehat H_\eta.
\end{equation}
In particular, this is a temperature shift with a quantified
entropy cost; no posterior label marginal is silently replaced.

\subsection{A Lipschitz bound in saturated coordinates}

Now set $\lambda=p/(8u)$ and $\tau=p/(8m)$ as in
\eqref{eq:spread-parameters}, and put
\begin{equation}\label{eq:saturation-parameter}
 a=\frac{\lambda}{4\theta}=\frac{p}{32B},\qquad
 z_i=\psi_a(y_i)=1-e^{-ay_i}.
\end{equation}

\begin{lemma}[Uniform transformed gradient bound]
\label{lem:gradient}
There is an $L$-Lipschitz function $g_\eta:[0,1]^d\to\R$ such that
$\widehat H_\eta(y)=g_\eta((\psi_a(y_i))_i)$, where
\begin{equation}\label{eq:gradient}
 L^2\le \frac{m\sqrt{q_{\rm sp}}}{a^2}
       =1024\,\frac{B^2m}{p^2}e^{-\tau/2}.
\end{equation}
\end{lemma}

\begin{proof}
For fixed $\pi$, the derivative of $\mathcal A_y(\pi)$ in $y_i$,
at almost every $y_i$, equals
\[
 d_i=\min\{Q_i^\pi(y_i/\theta),
                  \sqrt{q_{\rm sp}}e^{-2ay_i}\}.
\]
Since every witness has at most $m$ nonzero coordinates,
$\sum_i d_i\le\sum_i Q_i^\pi(y_i/\theta)\le m$. The inverse
change of variables has derivative $e^{ay_i}/a$, so
\[
 \sum_i\left|\partial_{z_i}
       \mathcal A_{\Psi^{-1}(z)}(\pi)\right|^2
 =a^{-2}\sum_i d_i^2e^{2ay_i}
 \le a^{-2}\sqrt{q_{\rm sp}}\sum_i d_i
 \le a^{-2}m\sqrt{q_{\rm sp}}.
\]
Thus each function
$z\mapsto\mathcal A_{\Psi^{-1}(z)}(\pi)-\eta K_\pi$
is $L$-Lipschitz on the open cube, with the same constant.
Their supremum is also $L$-Lipschitz. It extends uniquely and
continuously to the closed cube. There is no singularity in the
payoffs at $z_i=1$: their $i$th contribution is already constant
once $y_i\ge\theta u=B$.
\end{proof}

\begin{proposition}[Exponential moments at every temperature]
\label{prop:free-mgf}
For $Y=|X|$ with unconditional log-concave $X$, there is a universal
$C_F$ such that
\begin{equation}\label{eq:free-mgf}
 \log\E\exp\{t(\widehat H_\eta(Y)-\E\widehat H_\eta(Y))\}
 \le C_Ft^2\frac{B^2m}{p^2}e^{-\tau/2}
 \qquad(t\in\R).
\end{equation}
\end{proposition}
\begin{proof}
Apply \Cref{thm:saturated-mgf} to the function in
\Cref{lem:gradient}. One may take $C_F=1024C_T$.
\end{proof}

\subsection{From feasible witnesses to positive mean energy}

The next estimate combines the geometric and analytic parts of the
proof. Its first term is the available witness gain. The three
subtracted terms are, respectively, the probability cost of realizing
a witness, the cost of restricting the reference label law, and the
fluctuation cost supplied by transport. In the main application
$\alpha=1$, so the middle loss vanishes.

\begin{proposition}[Mean free-energy lower bound]\label{prop:free-mean}
Assume in addition that, for every label in the support of $\mu$,
\begin{equation}\label{eq:feasible-general}
 a_t([d])\ge\beta,\qquad \Prob(Y\ge a_t)\ge e^{-P_{\rm tail}}.
\end{equation}
With $\eta'=\eta+b$ one has
\begin{equation}\label{eq:free-mean}
 \E H_\eta(Y)\ge
 \theta\beta-\eta'\theta P_{\rm tail}
 -b\log\alpha^{-1}
 -C_F\frac{B^2m}{p^2\eta'}e^{-\tau/2}.
\end{equation}
If $\tau\ge\log4$, then $b\le32B/p$.
\end{proposition}
\begin{proof}
By \Cref{lem:tails}, each event $Y\ge\theta a_t$ has probability
at least $e^{-\theta P_{\rm tail}}$. On that event,
$C_t(Y)=\theta a_t([d])\ge\theta\beta$. Therefore
\begin{equation}\label{eq:annealed}
 \log\E e^{H_{\eta'}(Y)/\eta'}
 =\log\E_\mu\E e^{C_t(Y)/\eta'}
 \ge\frac{\theta\beta}{\eta'}-\theta P_{\rm tail}.
\end{equation}
On the other hand, \eqref{eq:Hcompare} and
\eqref{eq:free-mgf} at $t=1/\eta'$ give
\[
 \log\E e^{H_{\eta'}(Y)/\eta'}
 \le\frac{\E\widehat H_\eta(Y)+b\log\alpha^{-1}}{\eta'}
      +C_F\frac{B^2m}{p^2(\eta')^2}e^{-\tau/2}.
\]
Combine this with \eqref{eq:annealed}, multiply by $\eta'$, and
use $H_\eta\ge\widehat H_\eta$. Finally
$1-\sqrt{q_{\rm sp}}\ge1/2$ when $\tau\ge\log4$, so
$b\le4\theta/\lambda=32B/p$.
\end{proof}

\section{The signed comparison and its constants}\label{sec:signed}

We now use a uniform law $\mu$ on $n=e^p$ labels and the spread
condition \eqref{eq:S}, so $\alpha=1$. Define, for a fixed
$y\in\R_+^d$,
\begin{equation}\label{eq:width-y}
 w(y)=\E_\varepsilon\max_t
                    \left|\sum_{i\in A_t}\varepsilon_i y_i\right|.
\end{equation}
The next argument is deterministic in $y$.
A large positive gain for many labels does not by itself prevent
cancellation after signs are restored. The overlap estimate resolves
this problem: it yields a large subfamily for which each support
has substantial mass outside every other support. The corresponding
signed coefficient vectors are then separated, and Bernoulli
minoration applies.

\begin{lemma}[Many capped gains force signed width]\label{lem:signed}
For $\gamma>0$, let
$G_y=\{t:C_t(y)\ge4\gamma B\}$ and let $\delta$ be as in
\eqref{eq:overlap-tail}. If
$w(y)<c_B\gamma B$, then
\begin{equation}\label{eq:good-count}
 |G_y|<(1+2\delta n)e^{2m}.
\end{equation}
Consequently, with
\begin{equation}\label{eq:qstar}
 q_*=\frac{(1+2\delta n)e^{2m}}n,
\end{equation}
if $q_*<1$ and $B/\eta\le\log(1/q_*)$, then
\begin{equation}\label{eq:low-width-free}
 w(y)<c_B\gamma B
 \quad\Longrightarrow\quad
 H_\eta(y)\le4\gamma B+\eta\log2.
\end{equation}
\end{lemma}

\begin{proof}
For each fixed label $s$, \Cref{lem:overlap} with $D=A_s$ gives
at most $\delta n$ labels $t$ with $a_t(A_s)\ge\gamma u$.
On $G_y$ join $s,t$ if either $a_t(A_s)\ge\gamma u$ or
$a_s(A_t)\ge\gamma u$. The number of undirected edges is at most
$\delta n|G_y|$. Thus there is an independent set $V\subset G_y$
with
\[
 |V|\ge\frac{|G_y|}{1+2\delta n}.
\]
For distinct $s,t\in V$,
\[
 \sum_{i\in A_t\setminus A_s}y_i
 \ge C_t(y)-\theta a_t(A_s)\ge3\gamma B,
\]
and the same holds with $s,t$ reversed. The difference between the
coefficient vectors $(y_i\one_{A_t}(i))_i$ and
$(y_i\one_{A_s}(i))_i$ therefore has $\ell_1$ norm at least
$6\gamma B$ and support of size at most $2m$.
The signs all align with this difference on an event of probability
at least $2^{-2m}$. Hence its Bernoulli $L_{2m}$ norm is at least
$3\gamma B$.

If $|V|\ge e^{2m}$, \Cref{input:bernoulli} gives
$w(y)\ge3c_B\gamma B$, contradicting the hypothesis of the lemma.
This proves \eqref{eq:good-count}.

Since $C_t\le B$, splitting the partition function between $G_y$
and its complement yields
\[
 e^{H_\eta(y)/\eta}
 \le e^{4\gamma B/\eta}+q_*e^{B/\eta}.
\]
Under the stated temperature condition the second term is at most
one, and the right-hand side is at most
$2e^{4\gamma B/\eta}$. Taking its logarithm proves
\eqref{eq:low-width-free}.
\end{proof}

\subsection{A closed spread-witness theorem}

At this point all the necessary estimates are available. The theorem
below chooses their parameters together. On a first reading, the
important order is: fix the required gain and overlap levels; make
$m/p$ small; choose the shrinkage factor and temperature; finally
take $p$ large enough. This prevents a later choice from changing a
constant already used in the sparse reduction.

\begin{theorem}[Signed minoration from spread witnesses]
\label{thm:spread-minoration}
Fix $b_0=1/8$ and $\Lambda=3$. There are universal
$\epsilon_*>0$, $p_*<\infty$, and $c_*>0$ with the following
property. Let $Y=|X|$ for an unconditional log-concave $X$.
Suppose a family of $n=e^p$ labels, $p\ge p_*$, has supports and
witnesses satisfying
\begin{align}
 &\supp a_t\subset A_t,\quad |A_t|\le m\le\epsilon_*p,
       \quad m\ge1,\label{eq:spread-thm-support}\\
 &b_0u\le a_t([d])\le u,\quad
       \Prob(Y\ge a_t)\ge e^{-\Lambda p},\label{eq:spread-thm-tail}
\end{align}
where $u>0$, together with \eqref{eq:S} for the parameters
\eqref{eq:spread-parameters}. Then
\begin{equation}\label{eq:spread-thm-conclusion}
 \E_Y w(Y)\ge c_*u.
\end{equation}
\end{theorem}

\begin{proof}
\emph{Step 1: choose universal parameters.}
All constants will be selected before the source or family is used.
Let $C_F$ be the constant from \Cref{prop:free-mgf}, and choose
\begin{equation}\label{eq:gamma-kappa}
 \gamma=\frac{b_0}{64},\qquad \kappa=\frac{\gamma}{128}.
\end{equation}
Choose $\epsilon_*>0$ so small that
\begin{equation}\label{eq:epsilon-choice}
 \epsilon_*\le\frac{\gamma}{256},\qquad
 \epsilon_*\le\frac{1}{8\log4},\qquad
 C_F\kappa\epsilon_*\le\frac{b_0}{8}.
\end{equation}
Put
\begin{equation}\label{eq:theta-choice}
 A_\kappa=\kappa^{-1}+32,\qquad
 \theta=\min\!\left(\frac12,\frac{b_0}{8\Lambda A_\kappa}\right),
 \qquad B=\theta u,\qquad \eta=\frac{B}{\kappa p}.
\end{equation}
Finally choose
\begin{equation}\label{eq:pstar-choice}
 p_*\ge
 \max\!\left\{1,\frac{128\log3}{\gamma},
                  \frac{16\log2}{b_0\kappa}\right\}.
\end{equation}

\emph{Step 2: keep a fixed proportion of the witness gain.}
We have $\tau=p/(8m)\ge\log4$, so
$b\le32B/p$ and
\[
 \eta\le\eta'=\eta+b\le A_\kappa B/p.
\]
Apply \Cref{prop:free-mean} with
$\beta=b_0u$, $P_{\rm tail}=\Lambda p$, and $\alpha=1$.
The tail-feasibility loss is at most
\[
 \eta'\theta\Lambda p
 \le A_\kappa\Lambda\theta B\le b_0B/8.
\]
The transport loss is at most
\[
 C_F\frac{B^2m}{p^2\eta'}e^{-\tau/2}
 \le C_F\kappa B\frac mp e^{-\tau/2}
 \le b_0B/8.
\]
It follows that
\begin{equation}\label{eq:positive-free-energy}
 \E H_\eta(Y)\ge\frac{3b_0}{4}B
                   \ge\frac{b_0}{2}B.
\end{equation}

\emph{Step 3: ensure that small signed width forces small free energy.}
We next check the combinatorial temperature condition in
\Cref{lem:signed}. Since $m\le\epsilon_*p$ and
$q_{\rm sp}\le1$,
\[
 \delta=e^{mq_{\rm sp}-\gamma p/16}
 \le e^{-\gamma p/32}.
\]
Using $\gamma<1$ and \eqref{eq:epsilon-choice},
\begin{align*}
 q_*&=e^{2m}(e^{-p}+2\delta)
 \le3e^{2\epsilon_*p-\gamma p/32}\\
 &\le3e^{-3\gamma p/128}
 \le e^{-\gamma p/64}.
\end{align*}
The last inequality uses \eqref{eq:pstar-choice}.
Consequently
$\log(1/q_*)\ge\gamma p/64\ge\kappa p=B/\eta$.
Also
\[
 4\gamma B=\frac{b_0B}{16},\qquad
 \eta\log2\le\frac{b_0B}{16}.
\]
By \Cref{lem:signed}, on
$\{w(Y)<c_B\gamma B\}$ one therefore has
$H_\eta(Y)\le b_0B/8$. Since $H_\eta\le B$ everywhere,
\eqref{eq:positive-free-energy} implies
\[
 \frac{b_0B}{2}
 \le \frac{b_0B}{8}
       +B\,\Prob\{w(Y)\ge c_B\gamma B\}.
\]
\emph{Step 4: average the signed width.}
The probability is at least $3b_0/8$, hence at least $b_0/4$.
Thus
\[
 \E w(Y)\ge\frac{b_0c_B\gamma B}{4}
             =c_*u,\qquad
 c_*=\frac{b_0c_B\gamma\theta}{4}>0.
\]
Every displayed constant is universal.
\end{proof}

\section{Completion of the packing argument}\label{sec:completion}

We now connect the intermediate statements. The reduced theorem first
starts from actual joint-tail bodies and performs the two selections.
The proof of the main theorem then checks that the classical sparse
reduction produces exactly those bodies and that its approximation
error can be absorbed.

\subsection{A reduced theorem using actual joint-tail bodies}

\begin{theorem}[Minoration under the directed body alternative]
\label{thm:reduced}
There are universal $\zeta>0$, $q_0<\infty$, and $c_*>0$ such
that the following holds. Let $Y=|X|$ for an unconditional
log-concave vector, let $|\mathcal T|=N=e^q$ with $q\ge q_0$,
and suppose $|A_t|\le m\le\zeta q$. For $u>0$ put
$K_t=K_{A_t}(q,u)$ as in \eqref{eq:tail-body}.
If these bodies satisfy \eqref{eq:H}, then
\begin{equation}\label{eq:reduced-conclusion}
 \E_Y\E_\varepsilon\max_{t\in\mathcal T}
          \left|\sum_{i\in A_t}\varepsilon_iY_i\right|\ge c_*u.
\end{equation}
\end{theorem}

\begin{proof}
Use the constants of \Cref{thm:spread-minoration} and set
\begin{equation}\label{eq:zeta-q0}
 \zeta=\min(1/64,3\epsilon_*/8),\qquad
 q_0=\max(64,8p_*/3).
\end{equation}
\Cref{prop:common,cor:directed} give fixed witnesses
$a_t\in K_t$ and a subfamily $\mathcal T_1$ with
$p_1=\log|\mathcal T_1|\ge q/2$ satisfying
\eqref{eq:directed}. Apply \Cref{prop:spread} once, retain
$\mathcal S$, and delete its common core. Set
$p=\log|\mathcal S|$; then
\[
 p\ge3p_1/4\ge3q/8\ge p_*,
 \qquad m\le\zeta q\le\epsilon_*p.
\]
The restricted witnesses have mass at most $u$ and at least $u/8$:
for each label, take any other retained label in
\eqref{eq:directed}. Projection changes neither that directed
mass nor its lower bound.

Feasibility improves under projection, and
\[
 \Prob(Y_{\rm projected}\ge a_{t,\rm projected})
 \ge e^{-q}\ge e^{-3p}.
\]
The projected source is still unconditional log-concave.
The uniform law on the retained labels satisfies \eqref{eq:S}
with $p/(8m)$ and $p/(8u)$, by \Cref{prop:spread}.
All assumptions of \Cref{thm:spread-minoration} hold with
$b_0=1/8$ and $\Lambda=3$. Its conclusion gives the lower bound
for the projected signed width. \Cref{lem:projection} transfers
it to the original width.
\end{proof}

\subsection{Proof of the main theorem}

\begin{proof}[Proof of \Cref{thm:main}]
An unconditional log-concave law has mean zero and diagonal
covariance. Its affine support is a coordinate subspace: the affine
support is invariant under every coordinate sign change, and its
linear span is consequently generated by a subset of coordinate
axes. Delete zero coordinates and rescale the others diagonally to
make the law isotropic. Apply the inverse diagonal change to the
labels. Both $d_P$ and $\W_X$ are unchanged, so this normalization
loses nothing.

Choose the constants in the following order. First choose
$\epsilon_*,p_*,c_*$ as in \Cref{thm:spread-minoration};
then choose $\zeta,q_0$ by \eqref{eq:zeta-q0}. In the classical
reduction take
\begin{equation}\label{eq:rounding-final}
 h=\zeta/4,\qquad e=\frac{c_*}{32C_w}.
\end{equation}
Choose a universal $P_{\rm large}$ large enough for
\Cref{input:reduction}, for $P/4\ge q_0$, and for
\[
 \frac P4\le q:=\log\lceil e^{P/4}\rceil\le\frac P2
 \quad(P\ge P_{\rm large}).
\]
This order avoids any dependence of the later minoration constant
on a rounding parameter still to be chosen.

For $P\ge P_{\rm large}$, the easy branch of
\Cref{input:reduction} already gives the result. In the other
branch use its supports $A_t$ and coefficients $r_i$, and set
$Y_i=r_i|X_i|$. Coordinates with $r_i=0$ may be removed.
Let $m=\max_t|A_t|$. Then
$m\le hP\le\zeta q$ and $2m\le P$.

Fix distinct $s,t$. On $J=A_t\triangle A_s$ unconditionality
identifies the moment of the signed difference with the moment
having all coefficients positive:
\[
 \norm{\sum_{i\in J}r_iX_i}_P
 =\norm{\sum_i r_i(\one_{A_t}(i)-\one_{A_s}(i))X_i}_P
 \ge A/2.
\]
\Cref{input:witness} supplies a vector $v\ge0$ supported on $J$
such that
$\Prob(Y\ge v)\ge e^{-P}$ and $v(J)\ge A/(2C_w)$.
At least one of the two directed differences carries mass
$A/(4C_w)$. Restrict $v$ to that difference and shrink it by
$q/P$. Tail log-concavity gives probability at least $e^{-q}$,
and its mass is at least
\[
 \frac qP\,\frac{A}{4C_w}\ge\frac{A}{16C_w}.
\]
Scale down further, if necessary, to mass exactly
\begin{equation}\label{eq:u-final}
 u=\frac{A}{16C_w}.
\end{equation}
It follows that the actual joint-tail bodies
$K_t=K_{A_t}(q,u)$ satisfy the directed alternative \eqref{eq:H}.
No product of marginal probabilities is substituted here.

Apply \Cref{thm:reduced} and then the transfer estimate
\eqref{eq:reduction-transfer}:
\[
 \frac{c_*A}{16C_w}\le M_{\rm bin}
       \le2\W_X(T)+\frac{c_*A}{32C_w}.
\]
Thus $\W_X(T)\ge c_*A/(64C_w)$ in this branch.

For $1\le P<P_{\rm large}$, choose two distinct labels $s,t$.
Centering and the identity
$\max(U,V)=(U+V+|U-V|)/2$ give
\[
 \W_X(T)\ge\frac12\E|\ip{t-s}{X}|
 \ge\frac{A}{2C_{\rm reg}P_{\rm large}}
\]
by \eqref{eq:scalar}. Taking the minimum of this constant,
$c_*/(64C_w)$, and the easy-branch constant $c_{h,e}$ proves
\eqref{eq:main-conclusion}.
\end{proof}

\subsection{Where each structural assumption enters}

The proof uses log-concavity in the moment witness and sparse
reduction, in convexity and shrinkage of joint-tail bodies, and in
the source potential and conditional densities of the transport
argument. It uses unconditionality to represent independent signs,
to remove common coordinates by Jensen, and to make each
conditional transport map odd. The prefix and core selections
themselves are deterministic.

Negative association is not an assumption anywhere in these steps.
The product bound \eqref{eq:label-mgf-domination} concerns an
artificial distribution on the fixed witnesses, created by the
common-core optimization; it is not a product bound for the
coordinates of $X$. This distinction is the reason that the
argument, if verified, reaches beyond negative association.
The manuscript makes no claim for log-concave vectors without
unconditionality.

\section{Two consequences}\label{sec:consequences}

\subsection{An entropy bound in the centroid metric}

The first consequence expresses the result as a covering estimate.
The expected supremum bounds how many well-separated labels can fit
into the moment metric. This is a useful form when packing bounds
are to be inserted into a chaining argument.

For $p\ge1$ define the $L_p$ centroid body by its support function
\[
 h_{Z_p(X)}(t)=\norm{\ip{t}{X}}_p.
\]
In the nondegenerate case its polar $Z_p(X)^\circ$ is the unit ball
of this norm. Let $\mathcal N(T,rB)$ denote the least number of
translates of $rB$, with centers in $T$, covering $T$.

\begin{corollary}\label{cor:covering}
There is a universal $C$ such that, for every finite $T$ and every
unconditional log-concave $X$,
\begin{equation}\label{eq:centroid-cover}
 \mathcal N\bigl(T,C\W_X(T)Z_p(X)^\circ\bigr)\le e^p,
 \qquad p\ge1.
\end{equation}
For a degenerate law the assertion is understood on the quotient
by the kernel of the moment seminorm.
\end{corollary}

\begin{proof}
Take $C>1/c$, where $c$ is the constant in \Cref{thm:main}, and
let $w=\W_X(T)$. If $w>0$, choose a maximal subset $S\subset T$
whose distinct points have $L_p$ distance greater than $Cw$.
If $|S|\ge e^p$, \Cref{thm:main} gives
$w\ge\W_X(S)\ge cCw>w$, a contradiction.
Maximality supplies a cover of $T$ by $|S|<e^p$ balls of radius
$Cw$. If $w=0$, the two-point identity in the proof of the main
theorem shows that every difference of labels has zero moment
seminorm, so one ball suffices on the quotient.
\end{proof}

\subsection{Concentration of bounded magnitude statistics}

The second consequence is independent of the witness construction.
It applies to weighted sums of bounded coordinate responses, even
when the magnitudes are dependent.

\begin{corollary}\label{cor:statistics}
For any unconditional log-concave $X$, any $a_i>0$, and any
$v\in\R^d$, put
$F(X)=\sum_i v_i(1-e^{-a_i|X_i|})$. Then
\begin{equation}\label{eq:statistic-mgf}
 \log\E e^{t(F-\E F)}\le C_Tt^2\norm{v}_2^2
 \quad(t\in\R).
\end{equation}
In particular, for $r>0$ and $v\ne0$,
\begin{equation}\label{eq:statistic-tail}
 \Prob(|F-\E F|\ge r)
 \le2\exp\!\left(-\frac{r^2}{4C_T\norm{v}_2^2}\right).
\end{equation}
\end{corollary}

\begin{proof}
The function $g(z)=\ip{v}{z}$ is $\norm{v}_2$-Lipschitz.
Apply \Cref{thm:saturated-mgf} and optimize exponential Markov
in $t$ for each of the two tails.
\end{proof}

This consequence follows from the transport part alone. Its
normalization does not require independent coordinates, isotropy,
or bounds on the parameters $a_i$.

\appendix
\section{A compact dependency and parameter record}

The moment order and the three entropy parameters in the final
argument should not be identified:
\[
 P\quad\longrightarrow\quad
 q=\log\lceil e^{P/4}\rceil\quad\longrightarrow\quad
 p_1\ge q/2\quad\longrightarrow\quad p\ge3p_1/4.
\]
Here $P$ is the moment order in the original packing; $q$ is the
logarithm of the number of rounded labels; $p_1$ follows the
directed graph selection; and $p$ follows the one-time core
selection. Tail shrinkage changes the initial witness probability
from $e^{-P}$ to $e^{-q}$. Subsequent selections keep this actual
joint-tail probability, which is at least $e^{-3p}$.

The proof dependencies, in order, are:
\begin{enumerate}[leftmargin=*,itemsep=4pt]
\item The sparse reduction and sparse moment witness imply
\eqref{eq:H} with $u=A/(16C_w)$.
\item Prefix counting and finite-dimensional convex separation
give \eqref{eq:common-count}; one graph selection gives
\eqref{eq:directed}.
\item Maximizing $\Phi$ once gives \eqref{eq:S}; projection removes
the core without conditioning the source.
\item Clipping plus above-tangent transport gives
\eqref{eq:saturated-mgf}; entropy softening and the transformed
gradient estimate give \eqref{eq:free-mgf}.
\item Feasibility and the all-temperature bound give
\eqref{eq:positive-free-energy}; the overlap graph and Bernoulli
minoration give \eqref{eq:spread-thm-conclusion}.
\item The classical approximation error is absorbed using
\eqref{eq:rounding-final}.
\end{enumerate}

In particular, the former gap between an annealed partition function
and its mean logarithm is addressed by the explicit term
\[
 C_F\frac{B^2m}{p^2\eta'}e^{-p/(16m)}
\]
in \eqref{eq:free-mean}. The term comes from transport concentration
at $t=1/\eta'$, not from tail feasibility alone. Its validity rests
on the clipping and transport lemmas, whose proofs include the
regularization needed to keep the conditional maps continuous at
zero.

\newpage
\begingroup
\raggedright

\endgroup
\end{document}